\documentclass[12pt,a4paper]{amsart}
\usepackage[margin=2.5cm]{geometry}
\usepackage{tikz,tkz-graph}
\usepackage{float}
\usetikzlibrary{calc}
\usepackage{pgfplots}
\pgfplotsset{every axis/.append style={
                    axis x line=middle,    % put the x axis in the middle
                    axis y line=middle,    % put the y axis in the middle
                    axis line style={-,color=blue}, % arrows on the axis
                    xlabel={$x$},          % default put x on x-axis
                    ylabel={$y$},          % default put y on y-axis
            }}
\pgfplotsset{compat=1.13}

\usepackage[latin1]{inputenc}
\usepackage{amsmath}
\usepackage{amsthm}
\usepackage{amsfonts}
\usepackage{amssymb}
\usepackage{mathtools}
\usepackage{enumerate}
\usepackage{graphicx}
\usepackage{hyperref}
\usepackage{nicefrac}
\usepackage[all,cmtip]{xy}

\DeclareMathOperator{\Hom}{Hom}

\DeclareMathOperator{\Pic}{Pic}

\def\cS{{\mathcal S}}
\def\V{{\mathcal V}}

\def\ZZ{\mathbb{Z}}

\def\CC{\mathbb{C}}

\def\Q{\mathbb{Q}}
\def\QQ{\mathbb{Q}}

\def\cO{\mathcal{O}}

\def\div{\textrm{div}}

\def\ZK{Z_K}

\newtheorem{thm}{Theorem}[section]  %CHOOSE [chapter] or [section]
\newtheorem{prop}{Proposition}[section]
\newtheorem{cor}{Corollary}[section]

\newtheorem{def-lemma}{Definition-Lemma}[section]

\theoremstyle{remark}
\newtheorem{rem}{Remark}[section]

\theoremstyle{definition}

\newtheorem{exam}{Example}[section]

\makeatletter
\let\c@lemma\c@thm
\let\c@prop\c@thm
\let\c@propdef\c@thm
\let\c@proper\c@thm
\let\c@problem\c@thm
\let\c@conj\c@thm
\let\c@cor\c@thm
\let\c@rem\c@thm
\let\c@dfn\c@thm
\let\c@notation\c@thm
\let\c@exam\c@thm
\makeatother

\title[Delta invariant and genus of representable semigroups]
{Delta invariant of $\mathbb{Q}$-Cartier curve germs and the genus of representable numerical semigroups}

\author[Zs. Baja]{Zsolt Baja}
\address{Babe\c{s}-Bolyai University, Faculty of Mathematics and Computer Science,\newline \hspace*{6mm} 
Str. Mihail Kog\u{a}lniceanu nr. 1, 400084 Cluj-Napoca, Romania}
\email{zsolt.baja@ubbcluj.ro}

\author[T. L\'aszl\'o]{Tam\'as L\'aszl\'o}
\address{Babe\c{s}-Bolyai University, Faculty of Mathematics and Computer Science,\newline \hspace*{6mm} 
Str. Mihail Kog\u{a}lniceanu nr. 1, 400084 Cluj-Napoca, Romania}
\email{tamas.laszlo@ubbcluj.ro}

\author[A. N\'emethi]{Andr\'as N\'emethi}
\address{Alfr\'ed R\'enyi Inst. of Math.,  %\newline \hspace*{3mm}
Re\'altanoda utca 13-15, H-1053, Budapest, Hungary \newline \hspace*{3mm}
ELTE, % - Univ. of Budapest, Dept. of Geo. \newline \hspace*{3mm}
P\'azm\'any P\'eter s\'et\'any 1/A, 1117, Budapest, Hungary \newline \hspace*{3mm}
Babe\c{s}-Bolyai Univ., % Faculty of Math. and Comp. Sci. \newline \hspace*{3mm}
str. M. Kog\u{a}lniceanu 1, 400084 Cluj-Napoca, Romania \newline \hspace*{3mm}
BCAM, % - Basque Center for Applied Math.\newline \hspace*{3mm}
Mazarredo, 14 E48009 Bilbao, Basque Country, Spain}
\email{nemethi.andras@renyi.hu}

\thanks{All the authors are partially supported by NKFIH Grant `\'Elvonal (Frontier)' KKP 144148.  T.L. is supported by the `J\'anos Bolyai Research Scholarship' of the Hungarian Academy of Sciences. Zs.B. and T.L. acknowledge  the support of the project `Singularities and Applications' - CF 132/31.07.2023 funded by the European Union - NextGenerationEU - through Romania's National Recovery and Resilience Plan.}

\begin{document}

\begin{abstract}
In this article, first we give two formulae for the delta invariant of a complex curve singularity that can be embedded as a ${\mathbb Q}$-Cartier divisor in a normal surface singularity with rational homology sphere link. Next, we consider representable numerical semigroups, they are semigroups associated with normal weighted homogeneous surface singularities with rational homology sphere links (via the degrees of the homogeneous functions). We then prove that such a  semigroup can be interpreted as the value semigroup of a generic orbit (as a curve singularity)  given by the $\mathbb{C}^*$-action on the weighted homogeneous germ. Furthermore, we use the delta invariant formula to derive a combinatorially computable formula for the genus of representable semigroups. Finally, we 
characterize topologically those  representable semigroups which are  symmetric.
\end{abstract}

\maketitle

\section{Introduction}
\subsection{}
Numerical semigroups play a crucial role in various areas of the theory of complex curve and surface singularities. A typical example is when the finite number of gaps in a numerical semigroup provides a combinatorial interpretation of an analytic or topological invariant of the singularity (this number being the genus of the semigroup). One example is the delta invariant of an irreducible curve singularity, which is the genus of the value semigroup associated with the curve. 
While the semigroup interpretation of singularity invariants can be very valuable, the efficient computation of the main  characteristics of numerical semigroups  (like the 
Frobenius number, genus, etc.) is generally not a simple task. 
 In fact, it is generally rather difficult to find closed formulae.

In this context, the second and third authors \cite{stronglyflat} have recently proposed a 
research program which aims to link the theory of numerical semigroups and  surface singularities. 
Specifically, we consider a normal, weighted homogeneous surface germ and assume that its link is a rational homology sphere ($\mathbb{Q}HS^3$). We associate with this germ a numerical semigroup, $\mathcal{S}$, 
formed by the degrees of  the homogeneous elements  of the coordinate ring of functions, see \ref{ss:4.2}. Since the link is $\mathbb{Q}HS^3$, a result of Pinkham \cite{pinkham} implies that $\mathcal{S}$ is a topological invariant, it is computable from  the corresponding Seifert structure encoded by the canonical equivariant resolution graph. 
We say that a numerical semigroup is {\it representable}
if it can be associated with such a resolution graph in this way.

An important natural questions asks 
whether every  numerical semigroup can be represented in this way, or if not, what 
are the properties which obstruct the representability. 
Though this problem is still open,  the first and second authors characterized the representable semigroups as quotients of a special class of numerical semigroups, called the flat semigroups, see \cite{flat}. This is in some ways `analogue' to the fact that every numerical semigroup can be written as a quotient of a symmetric semigroup (cf. \cite{RGS,S}).  

The importance of the above-mentioned representability problem lies in the fact that various invariants of the representable numerical semigroups can be calculated or interpreted using methods from singularity theory,
thus providing answers to important questions originally  raised in the theory of numerical semigroups. 
For example, \cite{stronglyflat} gives an affirmative answer to the Frobenius problem by providing a closed formula for the Frobenius number of a representable semigroup.

The  present article provides a new step of this program.
We give a formula for the genus of a representable semigroup, which we identified with  the delta invariant of the generic orbit curve singularity that appears on the weighted homogeneous surface singularity. 

In the sequel we  summarize the results of this paper.

\subsection{}
In terms of results, the article can be divided into two parts. 

Section 3, as the first part of the article, is a generalization of the previous work of \cite{delta1,delta2}. Namely, in Theorem \ref{thm:delta} we give a formula (\ref{eq:delta1}) for the delta invariant of a curve singularity that can be embedded as a $\mathbb{Q}$-Cartier divisor in a normal surface singularity with $\mathbb{Q}HS^3$ link. This formula shows how the delta invariant depends on the embedded topology and the analytic type of the ambient surface. 

With the additional assumption that the ambient space  $(X,o)$ is $\mathbb{Q}$-Gorenstein,
we deduce formula (\ref{eq:delta2}), which in the second part 
will be the key ingredient in the study of representable semigroups.
%to the second part of the article. 
%Here, a special cycle emerges as the dominant feature of the embedded topological data (see Remark \ref{rem:delta}(2)), also playing a significant role in the calculation of the Frobenius number of representable semigroups, cf. \cite{stronglyflat}. 
In order to prove this second  formula, we need to establish some new  results 
valid for  $\mathbb{Q}$-Gorenstein normal surface singularities with $\mathbb{Q}HS^3$ link. A vanishing theorem is provided in Proposition \ref{thm:h1} and a duality result in Proposition \ref{prop:dual}, which establishes an intriguing duality property (Corollary \ref{cor:3.5}) for the equivariant geometric genus of such germs. 

The second part of the article begins with Section 4.
First, we briefly introduce the reader to the problem of representable semigroups. Then, in Theorem \ref{mainsmgps}, we prove that a representable numerical semigroup can be interpreted as the value semigroup of a  generic orbit curve singularity defined by the $\mathbb{C}^*$-action of the corresponding weighted homogeneous surface singularity. 
Then the results of the two parts combined provide the 
%Together with (\ref{eq:delta2}), this gives Corollary \ref{cor:4.3}, which consists of the 
desired formula for the genus of a representable semigroup. 

We conclude the article in Section 5, where (based on the just proved genus formula and an  earlier 
expression of the Frobenius number)  we characterize topologically 
(using the Seifert invariants of the link of the weightd homogeneous singularity) 
those representable semigroups which are symmetric.

\section{Preliminaries}\label{s:prel}

\subsection{Lattices associated with a resolution}
\label{sec:lattice}
We consider a complex normal surface singularity $(X,o)$ and one of its good resolutions $\pi:\tilde X\to X$ with dual resolution graph $\Gamma$ whose set of vertices is denoted by $\V$. Let $\{E_v\}_{v\in \V}$ be the irreducible components
of the exceptional set $\pi^{-1}(o)$. {\it We assume that the link $\Sigma$ is a 
rational homology sphere, $\mathbb{Q}HS^3$}, i.e.
$\Gamma$ is a connected tree and all $E_v$ are rational.

We define the lattice of integral cycles $L$ as  $H_2(\tilde X,\ZZ)$, endowed with the non-degenerate negative definite intersection form $(,)$. It is generated by the (classes of the) exceptional divisors $E_v$, $v\in \V$, that is, $L=\oplus_{v\in \V} \ZZ\langle E_v \rangle$. Note that in the homology exact sequence of the pair $(\tilde X, \Sigma)$ (where $\Sigma=\partial \tilde X$ is the link of the singularity) one has
$H_2(\Sigma,\ZZ)=0$, $H_1(\tilde X, \ZZ)=0$,  hence the exact sequence  has the form:
\begin{equation}
\label{eq:ses}
\array{ccccccccc}
0 & \to & L & \to & H_2(\tilde X,\Sigma,\ZZ) & \to & H_1(\Sigma,\ZZ) & \to & 0.\\
\endarray
\end{equation}
The dual lattice  $L':= \Hom(H_2(\tilde X,\ZZ),\ZZ)$ can be identified with $H_2(\tilde{X}, \Sigma, \ZZ)$ by the Lefschetz--Poincar\'e duality, hence \eqref{eq:ses} implies 
$L'/L\cong H_1(\Sigma,\ZZ)$. This finite group will be denoted by $H$ for simplicity. 

$L'$ embeds into the space of rational cycles $L_{{\mathbb Q}}:=
 L\otimes {\mathbb Q}$ and it can be identified  with $\{l'\in L_{{\mathbb Q}}\,:\, (l',L)_{{\mathbb Q}}\in \ZZ\}$, where $(\,,\,)_{{\mathbb Q}}$  is the extension of the intersection form to $L_{{\mathbb Q}}$. Hence, in the sequel we regard $L'$ as $\oplus_{v\in \V} \ZZ\langle E^*_v \rangle$,
 the lattice generated by the rational cycles  $E^*_v\in L_{{\QQ}}$,
$v\in \V$,  where $(E_u^*,E_v)_{{\QQ}}=-\delta_{u,v}$ (Kronecker delta) for any $u,v\in \V$.
 The inclusion $L\subset L'$  in the bases $\{E_v\}_v$ and $\{E_v^*\}_v$  is given by
$-M$, where $M$ is the intersection matrix of $L$, that is,
$E_v=-\sum_{u\in \V} (E_v,E_u) E_u^*$. Since $M$ is negative definite, the matrix  $-M^{-1}$
has  positive  entries, and the   $E^*_v$'s   are the columns of  $-M^{-1}$.

The dual base elements $E^*_v$, as rational divisors,  have the following
geometrical interpretation as well: consider $\gamma_v\subset \tilde X$ a curvette associated with $E_v$, that is, a smooth
irreducible curve germ in $\tilde X$ intersecting $E_v$ transversally. Then
$\pi^*\pi_*\gamma_v=\gamma_v+E_v^*$, where $\pi^*$ and $\pi_*$ are the natural divisorial operators. 

Let $K_{\tilde X}$ be the canonical divisor of  the smooth complex analytic surface $\tilde X$, 
and let $Z_K$  --- the {\it (anti-)canonical cycle} of $\pi$ --- be that rational cycle  supported on the exceptional locus which has the same numerical properties as $-K_{\tilde{X}}$.
Hence, it can be  determined combinatorially  by the
linear system of {\it adjunction formulae}
\begin{equation}
\label{eq:KX}
(Z_K,E_v)=(E_v,E_v)+2, \textrm{ for all } v\in \V.
\end{equation}
In particular, $Z_K\in L'$, and using \eqref{eq:KX} it can be also written as
\begin{equation}\label{eq:ZK}
\ZK-E=\sum_{v\in \V} (\delta_v-2) E^*_v.
\end{equation}
where $E=\sum_{v\in \V}E_v$ and $\delta_v$ is the valency of $v$ in $\Gamma$.

We say that $(X,o)$, or its topological type $\Sigma$, is \emph{numerically Gorenstein} if $Z_K\in L$. Note that this property is independent of the resolution and it  is equivalent with the topological triviality  of the line bundle $\Omega^2_{X\setminus o}$ of holomorphic $2$-forms on $X\setminus o$. $(X,o)$ is called \emph{Gorenstein} if $\Omega^2_{X\setminus o}$ is holomorphically trivial.
%, or equivalently, $K_X$ is a Cartier divisor at $o\in X$. 
More generally,  we say that $(X,o)$ is \emph{$\mathbb{Q}$-Gorenstein} if 
$(\Omega^2_{X\setminus o})^{\otimes r}$ is holomorphically trivial 
%Cartier at $o$ 
for some $r\in \mathbb{Z}_{>0}$. If $\Sigma$ is a $\mathbb{Q}HS^3$ then this property is equivalent with $\Omega^2_{\tilde X}=\cO_{\tilde X}(K_{\tilde X})\simeq \cO_{\tilde X}(-Z_K)$, where 
here $ \cO_{\tilde X}(-Z_K)$ denotes the 
`natural line bundle' associated with $(-Z_K)$.

Recall that for any $l'\in L'$ we can consider the unique line bundle ${\mathcal O}_{\tilde{X}}(l')$,
called the natural line bundle associated with $l'$, which has the property that 
$({\mathcal O}_{\tilde{X}}(l'))^{\otimes n}= {\mathcal O}_{\tilde{X}}(nl')$ for any 
$n$  with $nl'\in L$. (In the definition of the natural line bundle the fact that $\Sigma$ is a 
$\Q HS^3$ is necessary.)

\vspace{0.3cm}

For more regarding this section see eg. ~\cite{nbook}.

\subsection{Minimal \texorpdfstring{$H$}{H}--representatives and the generalized Laufer algorithm}\label{ss:HrepLip}
$L_{\QQ}$ has  a natural partial ordering defined coordinate-wise: 
for $l'_1,l'_2\in L_\QQ$ with $l'_i=\sum_v l'_{iv}E_v$ ($i=\{1,2\}$) one 
writes $l'_1\geq l'_2$  whenever  $l'_{1v}\geq l'_{2v}$ for all $v\in\V$. 
In particular, $l'$ is an effective rational cycle if $l'\geq 0$. We set also $\min\{l'_1,l'_2\}:= \sum_v\min\{l'_{1v},l'_{2v}\}E_v$.

Given an element $l'\in L'$ we denote by $[l']\in H$ its class in $H=L'/L$.
The lattice $L'$ admits a partition parametrized by the group $H$,  where for any $h\in H$ one sets 
$L'_h=\{l'\in L'\mid [l']=h\}\subset L'$. In particular, $L'_0=L$.
Given an $h\in H$ one can define the unique minimal representative $r_h$ of $h$ in 
$L'_h\cap L'_{\geq 0}$. By coordinates,    $r_h:=\sum_v l'_v E_v$ if and only if $[r_h]=h$   and 
$0\leq l'_v<1$ for avery $v$.

The set $\cS_\QQ:=\{l'\in L_\QQ \ | \ (l',E_v)\leq 0 \ \mbox{for all} \ v\in \V\}$ 
is a cone generated over $\QQ_{\geq 0}$ by $\{E^*_v\}_v$. Then $\cS':=\cS_\QQ\cap L'$ is a monoid of anti--nef rational cycles of $L'$ which is generated over $\mathbb{Z}_{\geq 0}$ by the cycles $E^*_v$. It is called the \emph{Lipman cone}. 
It  also admits a natural equivariant partition $\cS'_{h}=\cS'\cap L'_h$ indexed  by $H$. If $s_1,s_2\in \cS'_{h}$ then $\min\{s_1,s_2\}\in \cS'_h$. Furthermore, for any $h$ there exists a unique  \emph{minimal cycle} $s_h:=\min \{\cS'_{h}\}$, see eg. \cite{NemOSZ}. 

In the sequel we  describe the  \emph{generalized Laufer's algorithm} (\cite[Lemma 7.4]{NemOSZ},\cite{Laufer-rational}) that can be used to calculate $s_h$.

For any $l'\in L'$ there exists a unique
minimal element $s(l')\in \cS'$  such that $l'\leq s(l')$ and $[s(l')]=[l']\in H$. It  can be obtained by the following algorithm. Set $x_0:=l'$. Then one constructs  a computation sequence $\{x_i\}_i$ as follows. If $x_i$ is already constructed and $x_i\not\in\cS'$ then there exists some $E_{u_i}$ such that $(x_i,E_{u_i})>0$. In this case take $x_{i+1}:=x_i+E_{u_i}$ (for some choice of $E_{u_i}$). Then the procedure after finitely many steps stops, say at $x_t$, (ie. when 
$x_t\in \cS'$) and necessarily $x_t=s(l')$.

The point is  that $r_h\leq s_h$ and $s(r_h)=s_h$, hence the above algorithm produces $s_h$ from
$r_h$. However, in general $r_h\neq s_h$.
This fact does not contradict the minimality of $s_h$ in $\cS'_h$ since $r_h$ might not sit in $\cS'_h$.

Note also that  if $l'\in L'_{\leq 0}$, then $s(l')=s_{[l']}$. 

\subsection{The Riemann-Roch function}
For any fixed $h\in H$ we define the (equivariant) Riemann-Roch function as follows
\begin{equation}\label{eq:RR}
\chi_h:L\to \mathbb{Z}, \ \ \ \chi_h(l)=(Z_K-2s_h-l, l)/2.
\end{equation}
Note that $Z_K-2s_h\in L'$ is a `characteristic element' (cf. \cite[6.3.7]{nbook}) in the sense that one satisfies $(Z_K-2s_h-l,l)\in 2\ZZ$, hence $\chi_h$ is well defined. In fact, if we define the (full) Riemann-Roch function by $\chi:L'\to \ZZ$, $\chi(\ell'):=(Z_K-l',l')/2$ for any $l'\in L'$ then $\chi_h(l)=\chi(l+s_h)-\chi(s_h)$, hence $\{\chi_h\}_h$ is a conveniently chosen equivariant presentation of $\chi$ (restricted to the subsets $L'_h$). Clearly, if $h=0$ then $\chi_0(l)=\chi(l)$.
Recall also that  by the Riemann-Roch theorem
 $\chi(l)=h^0({\mathcal O}_l)-h^1({\mathcal O}_l)$
 for any effective $l\in L_{>0}$. 

\subsection{Universal abelian cover and the equivariant geometric genus}
\label{ss:uac}
We fix a good resolution $\pi$ of $(X,o)$ and consider $c:(Y,o)\to (X,o)$, the universal abelian covering of $(X,o)$ (see e.g. \cite[6.2.5]{nbook}). Let $\tilde Y$ be the normalized pull--back of $\pi$ and $c$, and denote by $\pi_Y$ and $\tilde{c}$
the induced maps that complete the following commutative diagram.
\begin{equation}
\label{eq:diagram}
\xymatrix{
\ar @{} [dr] | {\ }
\tilde{Y} \ar[r]^{\tilde{c}} \ar[d]_{\pi_Y} & \tilde{X} \ar[d]^{\pi} \\
Y \ar[r]_{c} & X
}
\end{equation}

The natural action of $H$ on $(Y,o)$ induces an action $h\cdot g(y)=g(h\cdot y)$ on $\cO_{Y,0}$, where $g\in \cO_{Y,0}$ and $h\in H$. This action decomposes $\cO_{Y,0}$ as $\oplus_{\lambda\in \hat{H}} (\cO_{Y,0})_{\lambda}$ according to the characters
$\lambda \in \hat{H}:={\rm Hom}(H,\CC^*)$, where
\begin{equation}
\label{eq:Heigenspaces}
(\cO_{Y,0})_{\lambda}:=\{g\in \cO_{Y,o} \mid g(h\cdot y)=\lambda (h)g(y),\ \forall y\in Y, h\in H\}.
\end{equation}
Note that there exists a natural isomorphism $\theta:H\to \hat{H}$ given by
$h\mapsto \exp(2\pi \sqrt{-1} (l',\cdot ))\in \hat{H}:= \Hom(H,\CC^*)$, where  $l' $ is any element of $L'$ with
$h=[l']$. 
%In order to simplify our notation, we write $(\cO_{Y,o})_h$ for $(\cO_{Y,o})_{\theta(h)}$
%(and similarly for any linear $H$--representation).

The $H$--eigenspace decomposition of $\tilde c_{*}(\cO_{\tilde Y})$ is given by (see~\cite{Nem-PS,Okuma-abelian,nbook})
$$\tilde c_{*}(\cO_{\tilde Y})=\bigoplus_{h\in H}\cO_{\tilde X}(-r_h) \   \ \mbox{with} \
\cO_{\tilde X}(-r_h)=(\tilde c_{*}(\cO_{\tilde Y}))_{\theta(h)},$$
where $\cO_{\tilde X}(l')$ is the unique  line bundle $\mathcal L$ on $\tilde X$ satisfying
$\tilde c^*\mathcal L=\cO_{\tilde Y}(\tilde c^*(l'))$ (see~\cite[3.5]{Nem-CLB}), where the pull--back $\tilde c^*(l')$ is an integral cycle in $\tilde Y$ for any $l'\in L'$.
(In other words, $\cO_{\tilde X}(l')$ is the natural line bundle associated with $l'$.)

Therefore, $H^1(\tilde X, \tilde c_{*}(\cO_{\tilde Y}))$ decomposes into $\bigoplus_{h\in H}H^1(\tilde X,\cO_{\tilde X}(-r_h))$ and the dimension of the $h$-part
\begin{equation}\label{eq:eqgen}
p_g(X,o)_{h}:=h^1(\tilde X,\cO_{\tilde X}(-r_h))
\end{equation}
is called the \emph{$h$-equivariant geometric genus} of $(X,o)$. 
%This can be expressed using the cycle $s_h$ by the formula
%\begin{equation}\label{eq:eqpg}
%-\chi(s_h).
%\end{equation}
Note that for $h=0$ we get $p_g(X,o)_{h}=p_g(X,o)$ the geometric genus of $(X,o)$, while $\sum_{h\in H}p_g(X,o)_{h}$ gives the geometric genus $h^1(\tilde{Y}, \mathcal{O}_{\tilde{Y}})$
of $(Y,o)$.

\subsection{{\bf Exact sequences and vanishing theorems}}\label{ss:vanishingth} 
In the sequel we list some exact sequences and vanishing theorems that will be used later in the proofs.

For  $l\in L_{>0}$ and $l'_1\in L'$ one can consider the exact sequence 
\begin{equation}
\label{eq:restriction2}
0\to \cO_{\tilde X}(-l-l'_1)\to \cO_{\tilde X}(-l'_1)\to \cO_{l}(-l'_1)\to 0.
\end{equation}
In the  generalized Laufer's algorithm, when  $x_{i+1}=x_i+E_{u_i}$
and  $(x_i,E_{u_i})>0$ (see~\ref{ss:HrepLip}), the choice $l_1'=x_i$ and $l=E_{u_i}$ in the cohomology exact sequence associated with \eqref{eq:restriction2}  applied  repeatedly gives
\begin{equation}\label{eq:sell'}
h^1(\cO_{{\tilde X}}(-l'))-\chi(l')=h^1(\cO_{{\tilde X}}(-s(l')))-\chi(s(l')).
\end{equation}
Furthermore, if $l'=r_h$ then $s(r_h)=s_h$, hence we get 
\begin{equation}\label{eq:rhsh}
p_g(X,o)_{h}-\chi(r_h)=
h^1(\cO_{{\tilde X}}(-r_h))-\chi(r_h)=h^1(\cO_{{\tilde X}}(-s_h))-\chi(s_h).
\end{equation}

Regarding the cohomology group $h^1({\tilde X}, {\mathcal L})$ there are several
 useful vanishing theorems.

\begin{thm}[\textbf{Grauert--Riemenschneider vanishing~\cite{GrRie,Laufer-rational,Ram}}]
\label{thm:GR}
For any $(X,0)$ (even if $\Sigma$ is not a $\mathbb{Q}HS^3$) and
 any $\mathcal L\in {\rm Pic}(\tilde X)$ with
$-c_1({\mathcal L})\in Z_K+ \cS'$ one has $h^1({\tilde X},{\mathcal L})=0$.
\end{thm}

\begin{thm}[\textbf{Generalized Grauert--Riemenschneider vanishing~\cite{EFBook}}]
\label{thm:GGR}  For any $(X,0)$ (even if $\Sigma$ is not a $\mathbb{Q}HS^3$),
 for any $\mathcal L\in {\rm Pic}(\tilde X)$ and for any $\Delta \in L_{\mathbb{Q}}$ with
 $\lfloor \Delta \rfloor=0$, if
$-c_1({\mathcal L})\in -\Delta + Z_K+ \cS_{\QQ}$, then $h^1({\tilde X},{\mathcal L})=0$.
\end{thm}

\section{The delta invariant of a $\mathbb{Q}$-Cartier curve}

\subsection{} Let $(X,o)$ be a normal surface singularity with $\mathbb{Q}HS^3$ link and $(C,o)\subset (X,o)$ a reduced curve on it. Assume further that \emph{$C$ is a $\mathbb{Q}$-Cartier divisor on $(X,o)$}, ie. $nC\subset X$ is a Cartier divisor for some $n>0$, or equivalently, the  class of $(C,o)$  has finite order in the local divisor class group $Cl(X,o)$. 

We fix a good embedded resolution $\pi:\tilde X \to X$ of the pair $C\subset X$. Let $\pi^*(C)=\tilde C+l'_C$ be the total transform of $C$, where $\tilde C$ is the strict transform and $l'_C\in L'$ is the rational cycle associated with the curve. This means 
that $\pi^*(C)$ is numerically trivial: the divisorial intersection of $\pi^*(C)$ with $ E_v$ is zero for every $v\in \mathcal{V}$. 

The divisor $\tilde C$ satisfies $c_1\cO_{\tilde X}(-\tilde C)=l'_C$ in $H^2(\tilde{X}, \mathbb{Z})\approx L'$ and, 
since $C$ is $\mathbb{Q}$-Cartier,
the linear equivalence $n(-\tilde C)\sim nl'_C$,  where $n$ is the order of $l'_C$ in $H$. Furthermore, since the link is a $\mathbb{Q}HS^3$, $\cO_{\tilde X}(-\tilde C)$ is unique with these properties, hence $\cO_{\tilde X}(-\tilde C)\simeq \cO_{\tilde X}(l'_C)$, see eg. \cite[Prop. 6.1.16]{nbook} (where $\cO_{\tilde X}(l'_C)$ is the natural line bundle associated with $l'_C$). 

Now, we regard  $(C,o)$ as an abstract reduced curve germ with irreducible decomposition $(C,o)=\bigcup_{i\in I}(C_i,o)$ ($I$ is finite). Let $\gamma_i:(\mathbb{C},0)\to (C,o), t\mapsto \gamma_i(t)$ be the normalization of the branch $(C_i,o)$ and for any $g\in \cO_{(C,o)}$ we consider its value $v(g):=(v_i(g))_{i\in I}$ defined by $v_i(g):=\textrm{ord}_t(g(\gamma_i(t)))$ for any $i\in I$. Then the multigerm $(C,o)^{\widetilde{}} :=(\mathbb{C},0)^{\vert I\vert}$ together with the map $(C,o)^{\widetilde{}}  \stackrel{\gamma:=(\gamma_i)_i}{\longrightarrow} (C,o)$ serve as the normalization. This defines the delta invariant of $(C,o)$ as $\delta(C,o):=\dim \gamma_*\cO_{(C,0)^{\widetilde{}}}/\cO_{C,0}$.

\subsection{The main formulae}
In this section we present two formulae for the delta invariant of $(C,o)$ assuming that it is embedded as a $\mathbb{Q}$-Cartier divisor in $(X,o)$. The first one will be a direct consequence of the work of \cite{delta1}, while the second one assumes that $(X,o)$ is a $\mathbb{Q}$-Gorenstein normal surface singularity. 

\begin{thm}\label{thm:delta}
Let $(X,o)$ be a  normal surface germ with $\mathbb{Q}HS^3$ link and $(C,o)\subset (X,o)$ be a $\mathbb{Q}$-Cartier curve on it and let $h_C:=[l'_C]$ be the class in $H$ associated with $(C,o)$. 

(a) Then the delta invariant of $(C,o)$ satisfies the following 
\begin{equation}\label{eq:delta1}
\delta(C,o)=\chi(-l'_C)-\chi(r_{-h_C})+p_g(X,o)_{-h_C}-p_g(X,o).
\end{equation}

(b) If we further assume that $(X,o)$ is $\mathbb{Q}$-Gorenstein, then we also have 
\begin{equation}\label{eq:delta2}
\begin{split}
 \delta(C,o) &= 
\chi_{[Z_K]+h_C}(Z_K+\ell'_C-s_{[Z_K+h_C]})+\chi(s_{[Z_K+h_C]}) \\ &-\chi(r_{[Z_K+h_C]})+p_g(X,o)_{[Z_K]+h_C}-p_g(X,o).
\end{split}
\end{equation}
\end{thm}
\vspace{0.1cm}

\begin{proof} Let us first prove part {\it (a) } of the theorem. The second part {\it (b)}
will be given after Corollary \ref{cor:3.5}, as a consequence of some preparation.

We know by \cite[Thm 1.3.(2)]{delta1} that, in general, we have $\delta(C,o)=h^1({\tilde X}, \cO_{{\tilde X}}(-\widetilde{C}))-p_g(X)$. Moreover, $h^1({\tilde X}, \cO_{{\tilde X}}(-\widetilde{C}))=h^1({\tilde X}, \cO_{{\tilde X}}(l'_C))$ 
since $(C,o)\subset (X,o)$ is $\mathbb{Q}$-Cartier. Therefore, by \cite[Thm 1.3(1)]{delta1} one has 
\begin{equation}\begin{split}\label{eq:uj}
\delta(C,o)  = & h^1({\tilde X},\cO_{{\tilde X}}(l'_C))-p_g(X,o) \\  =&
\chi(-l'_C)-\chi(s_{-h_C})+h^1({\tilde X}, \cO_{{\tilde X}}(-s_{-h_C})-p_g(X,o),
\end{split}\end{equation}
which, together with (\ref{eq:rhsh}), gives (\ref{eq:delta1}).
\end{proof}

\begin{rem}\label{rem:delta} 
(1) \ Formula (\ref{eq:delta1}) shows that the (a priori analytic) delta invariant of $(C,o)$ depends on some {\it embedded topological data} associated with the embedding 
$(C,o)\subset (X,o)$ and on the {\it $(C,o)$-independent analytic  data of the 
ambient space $(X,o)$} given by its equivariant geometric genera. 
For example, if two reduced curve germs $(C_1,o)$ and $(C_2,0)$ are embedded in $(X,o)$ in such a way that 
$h_{C_1}=h_{C_2}$, then 
$$\delta(C_1,o)-\chi(-l'_{C_1})=\delta(C_2,o)-\chi(-l'_{C_2}).$$
In the particular case when $(C,o)$ is Cartier we have  $h_C=0$, and (\ref{eq:delta1}) reduces to $\delta(C,o)=\chi(-l'_C)$, a formula which follows from  \cite{BGdelta} too, cf. \cite[3.4]{delta1}. \vspace{0.2cm}

(2) \  Our motivation behind  the second formula (\ref{eq:delta2})  is multi-layered.  Recent progress in the study of representable numerical semigroups (see sect. \ref{s:genorb}) have shown that the cycle $Z_K+\ell'_C-s_{[Z_K+h_C]}\in L$ 
gives the answer to several key questions. 

First, by \cite{stronglyflat}, in the case when $(X,o)$ is a weighted homogeneous singularities and $\tilde{C}$ is a transversal disc intersecting the central exceptional curve $E_0$  
corresponding to the star shaped graph (of the canonical resolution),  
the coefficient of the cycle $Z_K+\ell'_C-s_{[Z_K+h_C]}$ along $E_0$   gives the conductor (or the Frobenius number plus one) of the corresponding representable semigroup (for terminology see section \ref{s:genorb}).

This is also true when $(C,o)$ is a reduced (not necessarily irreducible) curve on a rational normal surface singularity $(X,o)$, \cite{CLMMN_sgp}. In this case, if we fix a good embedded resolution (with $\vert \tilde C \cap E_v\vert \leq 1$) of the pair $(C,o)\subset (X,o)$, then the restriction $(Z_K+\ell'_C-s_{[Z_K+h_C]})\vert_{I_C}$ to the support set $I_C:=\{v\in \V \ : \ \tilde C \cap E_v\neq\emptyset\}$ is equal to the conductor of the semigroup of $(C,o)$. \vspace{0.2cm}

(3) \ Another  appearance of the cycle 
$Z_K+\ell'_C-s_{[Z_K+h_C]}$ is in the delta invariant formula  from \cite{delta1,delta2} in the special case when $(X,o)$ is rational. This can be re-obtained from 
(\ref{eq:delta2}) as follows. Since $(X,o)$ is rational, 
$p_g=h^1(\tilde X, \cO_{\tilde X})=0$. On the other hand, by (\ref{eq:rhsh}) $(p_g)_{[Z_K]+h_C}=\chi(r_{[Z_K+h_C]})-\chi(s_{[Z_K+h_C]})$ since $h^1(\tilde X,\cO_{\tilde X}(-s_{h_C}))=0$ for $(X,o)$ rational (by Lipman's vanishing theorem \cite{Lipman}), cf. \cite[Prop. 4.3]{delta1}. Therefore, (\ref{eq:delta2}) reduces to $\delta(C,o)=\chi_{[Z_K]+h_C}(Z_K+\ell'_C-s_{[Z_K+h_C]})$ (cf. \cite[Thm. 1.1]{delta1}).
\end{rem}

Though in Remark \ref{rem:delta}(3)  the vanishing $h^1(\tilde X,\cO_{\tilde X}(-s_h))=0$ holds whenever $(X,o)$ is rational, this vanishing does  not happen  in general. However, in the sequel we will prove   that {\it this  vanishing is still valid whenever  $(X,o)$ is $\Q$-Gorenstein}.
Using this vanishing, we will prove for any $h\in H$ a cohomological duality result which connects $s_h$ with its `dual' $s_{[Z_K]-h}$ (see Proposition \ref{prop:dual}). This identity 
transforms formula (\ref{eq:delta1}) into  (\ref{eq:delta2}). 

\begin{prop}\label{thm:h1} Assume that $(X,o)$ is $\Q$-Gorenstein with $\Q HS^3$ link. Then for any $h\in H$ one has
\begin{equation}\label{eq:h1}
h^1(\cO_{{\tilde X}}(-Z_K+s_h))=0.
\end{equation}
\end{prop}

\begin{proof}
By the formal function theorem, (\ref{eq:h1}) is equivalent to the vanishing 
$h^1(\cO_{\ell}(-Z_K+s_{h}))=0$ for every $\ell\in L$ and $\ell\gg 0$. In fact, since $(X,o)$ is $\Q$-Gorenstein and the link is a $\Q HS^3$ one has that $\cO_{\tilde X}(K_{\tilde X}+Z_K)$ is trivial (cf. \cite[6.3.27]{nbook}), hence in the duality formulas one can replace $K_{\tilde X}$ with $-\ZK$. 
Therefore,  Serre duality implies 
\begin{equation}\label{eq:help1}
h^1(\cO_{\ell}(-Z_K+s_{h}))=h^0(\cO_{\ell}(\ell-s_h)).
\end{equation}

On the other hand, by the generalized Grauert--Riemenschneider vanishing theorem~\ref{thm:GGR}
and Serre duality we have 
\begin{equation}\label{eq:MAIN9}
h^0(\cO_{\ell}(\ell-r_h))=h^1(\cO_{\ell}(-Z_K+r_{-h}))=0.
\end{equation}
Furthermore, we claim that the map $\tau:H^1(\cO_{{\tilde X}}(-s_h))\to H^1(\cO_{{\tilde X}}(-r_h))$ is injective. Indeed, let us consider $\Delta_h:=s_h-r_h$. Then the exact sequence (cf.  (\ref{eq:restriction2}))
$0\to \cO_{{\tilde X}}(-s_h)\to \cO_{{\tilde X}}(-r_h)\to \cO_{\Delta_h}(-r_h)\to 0$ gives 
\begin{equation}\label{eq:UJ}
0\to H^0(\cO_{{\tilde X}}(-s_h))\xrightarrow{\alpha} H^0(\cO_{{\tilde X}}(-r_h)) \to H^0(\cO_{\Delta_h}(-r_h))\to H^1(\cO_{{\tilde X}}(-s_h))\xrightarrow{\tau} \dots .
\end{equation}
Since $s(r_h)=s_h$, we consider the computation sequence $\{r_h+x_i\}_{i=0}^{t}$ connecting $r_h$ with $s_h$, where $x_i\in L_{\geq 0}$, $x_0=0, x_t=\Delta_h$ and $x_{i+1}=x_i+E_{v(i)}$ for some $E_{v(i)}$ such that $(r_h+x_i,E_{v(i)})>0$. Then, by considering for any $i$ the long exact sequence
$$0\to H^0(\cO_{E_{v(i)}}(-r_h-x_i))\to H^0(\cO_{x_{i+1}}(-r_h))\to H^0(\cO_{x_{i}}(-r_h))\to \dots$$
and using that $H^0(\cO_{E_{v(i)}}(-r_h-x_i))=0$ by the Chern number inequality from construction of the sequence, one obtains 
 by induction that $H^0(\cO_{x_{i}}(-r_h))=0$. 
 In particular, $H^0(\cO_{\Delta_h}(-r_h))=0$, hence in (\ref{eq:UJ}) 
  $\alpha$ is an isomorphism and $\tau$ is injective.

Next, consider the following diagram
$$
\begin{array}{ccccccccc}
0 & \to & H^0(\cO_{{\tilde X}}(-s_h)) & \stackrel{\gamma}{\longrightarrow} &
H^0(\cO_{{\tilde X}}(\ell-s_h)) & \stackrel{k}{\longrightarrow} &
 H^0(\cO_{\ell} (\ell-s_h)) &
\stackrel{j}{\longrightarrow} & H^1(\cO_{{\tilde X}}(-s_h))\\
 & &
 \Big\downarrow\vcenter{%
 \rlap{$\scriptstyle{\alpha}$}}   & &   \Big\downarrow\vcenter{%
 \rlap{$\scriptstyle{\beta}$}} & & \Big\downarrow\vcenter{%
 \rlap{$\scriptstyle{i}$}} & & \Big\downarrow\vcenter{%
 \rlap{$\scriptstyle{\tau}$}}\\
0 & \to & H^0(\cO_{{\tilde X}}(-r_h)) & \stackrel{\omega}{\longrightarrow} &
 H^0(\cO_{{\tilde X}}(\ell-r_h)) & \to &
H^0( \cO_{\ell} (\ell-r_h)) & \stackrel{l}{\longrightarrow} & H^1(\cO_{{\tilde X}}(-r_h))
\end{array} 
$$
By (\ref{eq:MAIN9}) we know that $H^0( \cO_{\ell} (\ell-r_h))=0$, hence $\omega$ is an isomorphism. Since $\alpha$ is also an isomorphism, one follows that $\beta\circ\gamma$ is an isomorphism. Moreover, $\beta$ and $\gamma$ are injective, hence they are isomorphisms too. This implies that $k=0$ and $j$ is injective. The last, together with the injectivity of $\tau$, 
implies that $\tau\circ j=l\circ i$ is also injective. But $l\circ i=0$ since $H^0( \cO_{\ell} (\ell-r_h))=0$, we must therefore have $H^0( \cO_{\ell} (\ell-s_h))=0$. This, using (\ref{eq:help1}), implies (\ref{eq:h1}). 
\end{proof}

\begin{prop}\label{prop:dual}
Let $(X,o)$ be a $\mathbb{Q}$-Gorenstein singularity with $\mathbb{Q}HS^3$ link. Then for any $h\in H$ we have the following identity
\begin{equation}\label{eq:duality}
h^1(\cO_{{\tilde X}}(-s_{-h}))-\chi(s_{-h})=h^1(\cO_{{\tilde X}}(-s_{[Z_K]+h}))-\chi(s_{[Z_K]+h}).
\end{equation}
\end{prop}

\begin{proof}
    We consider the cycle $\ell':=Z_K-s_{-h}-\ell\in L'$ with $\ell\in L$. Hence $[\ell']=[Z_K]+h$.
Choose the coefficients of $\ell$ sufficiently large so that $\ell'\in L'_{\leq 0}$.
In this case $s(\ell')=s_{[Z_K]+h}$, cf.~\ref{ss:HrepLip}, and~\eqref{eq:sell'} implies 
\begin{equation}\label{eq:eq0}
h^1(\cO_{{\tilde X}}(-Z_K+s_{-h}+\ell))-\chi(Z_K-s_{-h}-\ell)= h^1(\cO_{{\tilde X}}(-s_{[Z_K]+h}))-\chi(s_{[Z_K]+h}).
\end{equation}
Now we consider the exact sequence 
$$0\to \cO_{{\tilde X}}(-Z_K+s_{-h})\to
\cO_{{\tilde X}}(-Z_K+s_{-h}+\ell)\to
\cO_{\ell}(-Z_K+s_{-h}+\ell)\to 0$$
which gives the long exact sequence 
{\footnotesize
\begin{equation}\label{eq:eq1}
\begin{array}{ccccccccc}
0 & \to & H^0(\cO_{{\tilde X}}(-Z_K+s_{-h})) & \to &
H^0(\cO_{{\tilde X}}(-Z_K+s_{-h}+\ell)) & \to &
 H^0(\cO_{\ell}(-Z_K+s_{-h}+\ell)) &\to & \\
 & \to & H^1(\cO_{{\tilde X}}(-Z_K+s_{-h})) & \to & H^1(\cO_{{\tilde X}}(-Z_K+s_{-h}+\ell)) & \to &
 H^1(\cO_{\ell}(-Z_K+s_{-h}+\ell)) &\to & 0.\\
\end{array} 
\end{equation}}
Then  we get 
\begin{equation}\label{eq:eq2}
h^1(\cO_{{\tilde X}}(-Z_K+s_{-h}+\ell))=h^1(\cO_{\ell}(-Z_K+s_{-h}+\ell))=h^0(\cO_{\ell}(-s_{-h})),
\end{equation}
where the first identity is implied by (\ref{eq:h1}) from Proposition \ref{thm:h1} and (\ref{eq:eq1}), while the second uses the Serre duality combined with the fact that $(X,o)$ is a $\mathbb{Q}$-Gorenstein singularity with $\mathbb{Q}HS^3$ link. The equation (\ref{eq:eq0}) combined with (\ref{eq:eq2}) gives 
\begin{equation}\label{eq:eq3}
h^0(\cO_{\ell}(-s_{-h}))-\chi(s_{-h}+\ell)= h^1(\cO_{{\tilde X}}(-s_{[Z_K]+h}))-\chi(s_{[Z_K]+h}).
\end{equation}
Finally, from the identities 
 $h^0(\cO_{\ell}(-s_{-h}))=\chi(\cO_{\ell}(-s_{-h}))+h^1(\cO_{\ell}(-s_{-h}))$, $\chi(\cO_{\ell}(-s_{-h}))=\chi(\ell)-(\ell, s_{-h})$ and $\chi(s_{-h}+\ell)=\chi(\ell)+\chi(s_{-h})-(\ell,s_{-h})$  
 applied to (\ref{eq:eq3}),  and  the formal function theorem
 $h^0(\cO_{\ell}(-s_{-h}))=h^0(\cO_{\tilde{X}}(-s_{-h}))$ applied for 
 $l\gg 0$ 
(see~\cite{GriffithsHarris}) imply the result.
\end{proof}
Combining (\ref{eq:duality}) with (\ref{eq:rhsh}) gives the following consequence. 
\begin{cor}\label{cor:3.5}
If $(X,o)$ is a $\mathbb{Q}$-Gorenstein singularity with $\mathbb{Q}HS^3$ link. Then for any $h\in H$ one has 
\begin{equation}\label{eq:pgdual}
p_g(X,o)_h-p_g(X,o)_{[Z_K]-h}=\chi(r_h)-\chi(r_{[Z_K]-h}).
\end{equation}
In particular, this identity says that the difference of the equivariant geometric genus associated with $h$ and its `dual' $[Z_K]-h$ is topological.
\end{cor}
\vspace{0.1cm}
\begin{proof}[Proof of Theorem \ref{thm:delta} part {\it (b)}]
Using the definition (\ref{eq:RR}) of the equivariant $\chi$-function, the expression (\ref{eq:uj})
rewrites as 
$$\delta(C,o)=\chi_{[Z_K]+h_C}(Z_K+l'_C-s_{[Z_K]+h_C})+\chi(s_{[Z_K]+h_C})-\chi(s_{-h_C})+h^1(\tilde X,\mathcal O_{\tilde X}(-s_{-h_C}))-p_g(X,o).$$
Then using Proposition \ref{prop:dual} and (\ref{eq:rhsh}) the formula follows.
\end{proof}

\subsection{Example}
We consider the Brieskorn hypersurface singularity $(X,0)=(\{x^4-y^6+z^5=0\},0)\subset (\mathbb{C}^3,0)$ and the reduced curve $(C,0)$ on 
$(X,0)$ given  by the ideal $(x^2+y^3,z)$. Note that $(C,0)$ is just the ordinary cusp with $\delta(C,0)=1$. The aim of this example is to obtain  $\delta(C,0)=1$ using  (\ref{eq:delta2}). 

First, note that $(C,0)$ is a $\mathbb{Q}$-Cartier divisor, since $5C$ is cut out by the function $x^2+y^3$ on $(X,0)$. Furthermore, we have 
(e.g. from the Newton diagram) that $p_g=6$. A good embedded resolution of $(C,0)\subset (X,0)$ is encoded by the following graph:
\vspace{0.2cm}
\tikzset{every picture/.style={line width=0.75pt}}         
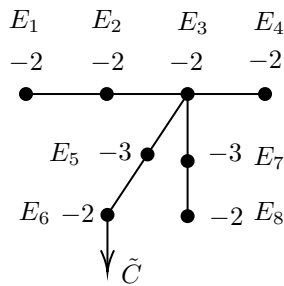
\begin{center}
\begin{tikzpicture}[x=0.75pt,y=0.75pt,yscale=-1,xscale=1]
\draw  [fill={rgb, 255:red, 0; green, 0; blue, 0 }  ,fill opacity=1 ] (247,119) .. controls (247,117.34) and (248.34,116) .. (250,116) .. controls (251.66,116) and (253,117.34) .. (253,119) .. controls (253,120.66) and (251.66,122) .. (250,122) .. controls (248.34,122) and (247,120.66) .. (247,119) -- cycle ;

\draw  [fill={rgb, 255:red, 0; green, 0; blue, 0 }  ,fill opacity=1 ] (287.33,119) .. controls (287.33,117.34) and (288.68,116) .. (290.33,116) .. controls (291.99,116) and (293.33,117.34) .. (293.33,119) .. controls (293.33,120.66) and (291.99,122) .. (290.33,122) .. controls (288.68,122) and (287.33,120.66) .. (287.33,119) -- cycle ;

\draw    (250,119) -- (369.33,119) ;

\draw  [fill={rgb, 255:red, 0; green, 0; blue, 0 }  ,fill opacity=1 ] (366.33,119) .. controls (366.33,117.34) and (367.68,116) .. (369.33,116) .. controls (370.99,116) and (372.33,117.34) .. (372.33,119) .. controls (372.33,120.66) and (370.99,122) .. (369.33,122) .. controls (367.68,122) and (366.33,120.66) .. (366.33,119) -- cycle ;

\draw  [fill={rgb, 255:red, 0; green, 0; blue, 0 }  ,fill opacity=1 ] (327.67,119) .. controls (327.67,117.34) and (329.01,116) .. (330.67,116) .. controls (332.32,116) and (333.67,117.34) .. (333.67,119) .. controls (333.67,120.66) and (332.32,122) .. (330.67,122) .. controls (329.01,122) and (327.67,120.66) .. (327.67,119) -- cycle ;

\draw    (330.67,119) -- (330.67,180.44) ;

\draw    (330.67,119) -- (290.67,179.78) ;
 
\draw  [fill={rgb, 255:red, 0; green, 0; blue, 0 }  ,fill opacity=1 ] (307.67,149.39) .. controls (307.67,147.73) and (309.01,146.39) .. (310.67,146.39) .. controls (312.32,146.39) and (313.67,147.73) .. (313.67,149.39) .. controls (313.67,151.05) and (312.32,152.39) .. (310.67,152.39) .. controls (309.01,152.39) and (307.67,151.05) .. (307.67,149.39) -- cycle ;

\draw  [fill={rgb, 255:red, 0; green, 0; blue, 0 }  ,fill opacity=1 ] (287.67,179.78) .. controls (287.67,178.12) and (289.01,176.78) .. (290.67,176.78) .. controls (292.32,176.78) and (293.67,178.12) .. (293.67,179.78) .. controls (293.67,181.43) and (292.32,182.78) .. (290.67,182.78) .. controls (289.01,182.78) and (287.67,181.43) .. (287.67,179.78) -- cycle ;

\draw  [fill={rgb, 255:red, 0; green, 0; blue, 0 }  ,fill opacity=1 ] (327.67,180.44) .. controls (327.67,178.79) and (329.01,177.44) .. (330.67,177.44) .. controls (332.32,177.44) and (333.67,178.79) .. (333.67,180.44) .. controls (333.67,182.1) and (332.32,183.44) .. (330.67,183.44) .. controls (329.01,183.44) and (327.67,182.1) .. (327.67,180.44) -- cycle ;

\draw  [fill={rgb, 255:red, 0; green, 0; blue, 0 }  ,fill opacity=1 ] (327.67,152.72) .. controls (327.67,151.07) and (329.01,149.72) .. (330.67,149.72) .. controls (332.32,149.72) and (333.67,151.07) .. (333.67,152.72) .. controls (333.67,154.38) and (332.32,155.72) .. (330.67,155.72) .. controls (329.01,155.72) and (327.67,154.38) .. (327.67,152.72) -- cycle ;

\draw    (290.67,179.78) -- (290.67,206.56) ;
\draw [shift={(290.67,208.56)}, rotate = 270] [color={rgb, 255:red, 0; green, 0; blue, 0 }  ][line width=0.75]    (10.93,-3.29) .. controls (6.95,-1.4) and (3.31,-0.3) .. (0,0) .. controls (3.31,0.3) and (6.95,1.4) .. (10.93,3.29)   ;

% Text Node
\draw (240,96.67) node [anchor=north west][inner sep=0.75pt]  [font=\footnotesize] [align=left] {$-2$};
% Text Node
\draw (280.67,97) node [anchor=north west][inner sep=0.75pt]  [font=\footnotesize] [align=left] {$-2$};
% Text Node
\draw (320,96.67) node [anchor=north west][inner sep=0.75pt]  [font=\footnotesize] [align=left] {$-2$};
% Text Node
\draw (359.33,96) node [anchor=north west][inner sep=0.75pt]  [font=\footnotesize] [align=left] {$-2$};
% Text Node
\draw (284.33,142.33) node [anchor=north west][inner sep=0.75pt]   [align=left] {{\footnotesize $-3$}};
% Text Node
\draw (339.33,142.67) node [anchor=north west][inner sep=0.75pt]   [align=left] {{\footnotesize $-3$}};
% Text Node
\draw (340,174.67) node [anchor=north west][inner sep=0.75pt]  [font=\footnotesize] [align=left] {$-2$};
% Text Node
\draw (265.67,173.33) node [anchor=north west][inner sep=0.75pt]  [font=\footnotesize] [align=left] {$-2$};
% Text Node
\draw (296,200.33) node [anchor=north west][inner sep=0.75pt]  [font=\footnotesize] [align=left] {$\displaystyle \tilde{C}$};
% Text Node
\draw (240,75) node [anchor=north west][inner sep=0.75pt]   [align=left] {{\footnotesize $E_1$}};
% Text Node
\draw (281,75) node [anchor=north west][inner sep=0.75pt]   [align=left] {{\footnotesize $E_2$}};
% Text Node
\draw (324,76) node [anchor=north west][inner sep=0.75pt]   [align=left] {{\footnotesize $E_3$}};
% Text Node
\draw (362,76) node [anchor=north west][inner sep=0.75pt]   [align=left] {{\footnotesize $E_4$}};
% Text Node
\draw (260,142) node [anchor=north west][inner sep=0.75pt]   [align=left] {{\footnotesize $E_5$}};
% Text Node
\draw (245,172) node [anchor=north west][inner sep=0.75pt]   [align=left] {{\footnotesize $E_6$}};
% Text Node
\draw (362,144) node [anchor=north west][inner sep=0.75pt]   [align=left] {{\footnotesize $E_7$}};
% Text Node
\draw (362,172) node [anchor=north west][inner sep=0.75pt]   [align=left] {{\footnotesize $E_8$}};
\end{tikzpicture}
\end{center}
For simplicity, we write $l'=(l'_1,\dots,l'_8)$ for a cycle $l'=\sum_{i=1}^8 l'_i E_i$. Then, by simple calculations we deduce that 
\begin{center}
$Z_K=(8,16,24,12,10,5,10,5)$ and $l'_C=E_6^*=(2,4,6,3,13/5,9/5,12/5,6/5)$. 
\end{center}
Therefore, $[Z_K]+h_C=h_C=[E_6^*]$ in $H=\mathbb{Z}_5$ and it is represented by $$r_{h_C}=(0,0,0,0,3/5,4/5,2/5,1/5).$$ The generalized Laufer algorithm (see \ref{ss:HrepLip}) from $r_{h_C}$ to $s_{h_C}$ shows that $\chi(r_{h_C})-\chi(s_{h_C})=2$, moreover, $s_{h_C}=E_6^*$. Therefore, the first term of (\ref{eq:delta2}) simplifies to $\chi_{h_C}(Z_K)=(Z_K,-E_6^*)=5$, the $E_6$-coefficient of $Z_K$. Finally, one can calculate the equivariant geometric genus $p_g(X,0)_{h_C}=4$ (eg. by the polynomial part of the equivariant Poincar\'e series associated with $(X,0)$, see \cite{LN14}). 
So, in summary, we get $\delta(C,0)=5-2+4-6=1$, which supports the formula. 

\section{Generic orbits and representable semigroups of weighted homogeneous germs}\label{s:genorb}
\subsection{Preliminaries on representable semigroups}
\subsubsection{\bf Notations and some preliminary  facts}\label{sss:facts}
Let $(X,0)$ be a normal weighted homogeneous surface singularity, which is defined as the germ at the origin of a normal affine surface $X$ with a good and effective  $\mathbb{C}^*$-action. In particular,  its affine coordinate ring is $\mathbb{Z}_{\geq 0}$--graded: $R_X=\oplus_{\ell\geq 0} R_{X,\ell}$. 

We consider the smooth compact curve $E_0:=(X\setminus\{0\})/\mathbb{C}^*$ and by $T$ we denote the closure of the graph of the map $X\setminus \{0\}\to E_0$ in $X\times E_0$. Then the first projection $T\to X$ is a modification of $(X,0)$, while the second projection $T\to E_0$ is a Seifert line bundle with zero section $E_0$. $T$ has at most a finite number of cyclic quotient singularities, which sit  at the intersection of $E_0$ with the singular fibers. After resolving in a minimal way these singularities, we get the {\em canonical (equivariant)  resolution} $\pi:\widetilde X\to X$.
The exceptional divisor $\pi^{-1}(0)$ is a normal crossing divisor and only the central curve $E_0$ can have a self-intersection number $-1$ and positive genus.

Let $\Gamma$ be the dual resolution graph of the canonical resolution $\pi$. Then, by
\cite{OW} $\Gamma$ is a `star-shaped' graph with a central vertex $v_0$ and $d\ge 0$ legs connected to it. (A leg is  a chain of vertices that corresponds to the resolution of a cyclic quotient singularity of $T$.)

The $\mathbb{C}^*$-action induces an Seifert $S^1$--action on the link $\Sigma$ of the singularity, in particular 
$\Sigma$ is a  Seifert 3-manifold with a negative definite intersection matrix. This Seifert 3-manifold 
is characterized by its normalized Seifert invariants  $Sf=(-b_0,g;(\alpha_i,\omega_i)_{i=1}^d)$, defined as follows. Each leg is determined by a pair of integers  $(\alpha_i,\omega_i)$,
where $0<\omega_i <\alpha_i$ and $\gcd(\alpha_i,\omega_i)=1$. The $i^{\mathrm{th}}$ leg has $\nu_i$ vertices, say $v_{i1},\ldots, v_{i\nu_i}$ ($v_{i1}$ is connected to the central vertex $v_0$)
with self-intersection numbers (or Euler decorations) 
$-b_{i1},\ldots, -b_{i\nu_i}$.  These can be determined by the Hirzebruch--Jung (negative) continued fraction expansion
$$ \alpha_i/\omega_i=[b_{i1},\ldots, b_{i\nu_i}]=
b_{i1}-1/(b_{i2}-1/(\cdots -1/b_{i\nu_i})\dots) \ \  \ \ (b_{ij}\geq 2).$$
All these vertices (except $v_0$) have genus--decorations zero. The central vertex $v_0$ corresponds to the central genus $g$ curve $E_0$ with self-intersection number $-b_0$.

In particular, $g=0$ if and only if $\Sigma$ is a $\mathbb{Q}HS^3$, a property which will be assumed in the sequel. In this case, we omit the genus $g$ from the notation and simply write $Sf=(-b_0;(\alpha_i,\omega_i)_{i=1}^d)$.

\subsection{\bf Representable numerical semigroups \cite{stronglyflat,flat}}\label{ss:4.2}
We consider the numerical semigroup $\mathcal{S}_{(X,0)}:=\{\ell\in\mathbb{Z}_{\geq 0} | R_{X,\ell}\neq 0\}$. 
If  $\Sigma$ is $\mathbb{Q}HS^3$ then  \cite{pinkham} shows  that $\dim(R_{X,\ell})$ is 
topological, namely it equals 
$\max\{0,1+N(\ell)\}$, where $N(\ell)$
is  the quasi-linear function 
\begin{equation}\label{defN}
N(\ell):=b_0\ell-\sum_{i=1}^d\Big\lceil \frac{\ell \omega_i}{\alpha_i}\Big\rceil.
\end{equation}
Hence, in this case the semigroup $\cS_{(X,0)}=\{\ell\in\mathbb{Z} \ | \ N(\ell)\geq 0\}$ is topological and it can be  described by Seifert invariants.  
For this reason, we denote it by $\cS_{\Gamma}$. 

We say that a numerical semigroup $\cS\subset \mathbb{N}$ is {\it representable} if it can be realized as $\cS_\Gamma$ for some canonical equivariant resolution graph $\Gamma$  as above. 

\subsection{The generic orbit curve and its semigroup}

Let $(X,0)$ be a normal weighted homogeneous surface singularity with $\mathbb{Q}HS^3$ link. (We fix a good $\mathbb{C}^*$-action if it is not unique.) As in \ref{sss:facts} consider the canonical equivariant resolution associated with the $\mathbb{C}^*$-action and lift the $\mathbb{C}^*$ orbits to the resolution space. There are special orbits corresponding to the singular fibers (the strict transform of any of these  intersects an  
end vertex $E_{i\nu_i}$)  and a 1-parameter family of generic orbits 
(their strict transforms intersect $E_0$ in the points 
$E_0\setminus \cup_i E_{v_{i1}}$). We replace the resolution of the affine 
surface by a small convenient neighborhood of the exceptional curve $E$ and we denote it by $\widetilde X$,  and let $\pi:(\widetilde X,E) \to (X,0)$ be the 
corresponding resolution map. 

Let $\widetilde C^{gen}$ be the intersection of a lifted generic orbit with $\widetilde X$ . It is a special cut of the central curve $E_0$ at a smooth point of $\pi^{-1}(0)$.  

In fact, the irreducible reduced germ $(C^{gen},0):= \pi(\widetilde{C}^{gen})$ is 
$\mathbb{Q}$-Cartier in $(X,0)$ whose total transform is  $\widetilde C^{gen}+E^*_0$, see eg. \cite[Ex. 6.2.13]{nbook}. This means that the line bundle $\cO_{{\widetilde X}}(-\widetilde C^{gen})$ is a `natural line bundle' in $\Pic(\widetilde X)$, ie. 
\begin{equation}
\cO_{{\tilde X}}(-\tilde C^{gen})\simeq \cO_{{\tilde X}}(E_0^*).
\end{equation}
In particular, all the generic orbits in $\tilde{X}$ define the same line bundle.
\vspace{0.5cm}

Now, we look at the (abstract) curve germ $(C^{gen},0)$ and consider the discrete valuation $v:\bar \cO_{C^{gen}}=\mathbb{C}\{t\}\to \mathbb{Z}_{\geq 0}\cup \infty$ defined by the normalization $\bar C^{gen}\to C^{gen}$. Then one can consider the value semigroup of $(C^{gen},0)$ defined by 
$$\cS_{C^{gen}}:=\{\ell \in \mathbb{Z}_{\geq 0} \ : \  \mbox{there exists} \ f\in \cO_{C^{gen}} \ \mbox{with} \ v(f)=\ell\}.$$
The next theorem claims that the value semigroup of the generic orbit curve (regarded as an abstract curve)
is the same as the semigroup $\cS_\Gamma$ associated with $\Gamma$.  

\begin{thm}\label{mainsmgps}
The numerical semigroup $\cS_{C^{gen}}$ associated with the germ of the irreducible curve $(C^{gen},0)$ coincides with $\cS_{\Gamma}$.   
\end{thm}

\begin{proof}
 We consider the canonical equivariant resolution $\pi$ as a good embedded resolution for $(C^{gen},0)\subset (X,0)$. We also write $p:=\tilde{C}^{gen}\cap E_0\in \tilde{X}$. 
 
 Recall that $\ell \in \cS_{C^{gen}}$ if and only if there is a function $f\in \cO_{C^{gen}}$ with $v(f)=\ell$, where $v$ is the valuation induced by the normalization of $(C^{gen},0)$. This condition is equivalent to the existence of a function $f\in \cO_{(X,0)}$ such that $(\div(f\circ \pi),\tilde C^{gen})_{\tilde{X}}=\ell$, where $\div(f\circ \pi):=(f)_{\Gamma}+\tilde C_f$ is the total transform,  $(f)_{\Gamma}$ is its part supported on $E=\pi^{-1}(0)$ and $\tilde C_f$ is the strict transform, and $(\ ,\, )_{\tilde{X}}$ denotes the divisorial intersection on $\tilde{X}$.
In particular, if the $E_0$ coefficient of $(f)_\Gamma$ is $m_{E_0}(f)$, then 
$\ell=v(f)=m_{E_0}(f)+(\tilde{C}_f, \tilde{C}^{gen})_p$, where the last term is the intersection multiplicity at $p\in\tilde{X}$. 

On the other hand, $\cS_{(X,0)}=\{\ell\in\mathbb{Z}_{\geq 0} | R_{X,\ell}\neq 0\}$ has the following reinterpretation. First note that if $f\in R_{X, \ell}$ is a homogeneous function, then 
$m_{E_0}(f)$ is exactly its homogeneous degree $\ell$. Hence  $\cS_{(X,0)}
= \{ m_{E_{0}}(f)\ |\ f \ \mbox{is homogeneous}\}$. 

The inclusion $\cS_{(X,0)}\subset  \cS_{C^{gen}}$ is rather direct. Let $f_{\ell}\in  R_{X,\ell}$
be a nonzero element of homogeneous degree $\ell$, and let $\tilde{C}_{f_{\ell}}$ be its strict transform. 
Then we can choose the generic orbit $\tilde{C}^{gen}$ in such a way that 
it does not intersect (along $E_0$) any of the strict transforms $\tilde{C}_{f_{\ell}}$, hence $(\tilde{C}_{f_{\ell}}, \tilde{C}^{gen})_p=0$.
Then $v$ associated with  $\tilde{C}^{gen}$ has value $\ell=m_{E_0}(f_\ell)=v(f_\ell)$.

Next, let us fix $\tilde{C}^{gen}$  and 
assume that $\ell\in  \cS_{C^{gen}}$. Choose  a function $f\in \cO_{(X,0)}$ such that 
$\ell=v(f)$. Write $f$ as a sum of  its homogeneous components $f=\sum _{j\geq 0} f_j$, where 
each $f_j \in R_{X,j}$ is nonzero. Then, for any fixed $j$, the strict transform $\tilde{C}_{f_j}$
of $f_j$ is a finite union of $\mathbb{C}^*$-orbits  intersected with $\tilde{X}$. 
It might contain the fixed orbit $\tilde{C}^{gen}$ or not. 
Therefore, $v(f_j)=j+(\tilde{C}_{f_j}, \tilde{C}^{gen})_p$ equals $j$ if 
$\tilde{C}_{f_j}$ does not contain the fixed orbit $\tilde{C}^{gen}$, and it is $\infty$ otherwise. 
In particular, all the possible finite values $\{v(f_j)\}_j$ are distinct, and 
$\ell=v(f)=\min \{v(f_j)\ |\  v(f_j)<\infty\}$. This shows that $\ell=v(f)=v(f_j)=j$
for some $f_j\in R_{(X,0)}$, hence $\ell\in \cS_{(X,0)} $. 
\end{proof}

\subsection{The genus of a representable semigroup} For a numerical semigroup $\cS$ with 
$|\mathbb{N}\setminus \cS| <\infty$, the number of gaps $\mathbb{N}\setminus \cS$
it is called the genus and is denoted by $g(\cS)$. 

In the previous section, we have proved that every representable semigroup can be viewed as the value semigroup of the generic orbit curve on a weighted homogeneous surface singularity  $(X,0)$. Therefore, Theorem \ref{thm:delta}  provides the following formula.  
\begin{cor}\label{cor:4.3}
Let $\cS$ be a representable semigroup and we consider one of its representatives $\Gamma$ and the associated cycles as in Sect. \ref{s:prel}. Then we have 
\begin{equation}\label{eq:genus}
g(\cS)=\chi_{[Z_K]+h_0}(Z_K+E_0^*-s_{[Z_K]+h_0})+\chi(s_{[Z_K]+h_0})-\chi(r_{[Z_K]+h_0})+(p_g)_{[Z_K]+h_0}-p_g.
\end{equation}
\end{cor}

\begin{rem}
We wish to emphasize that in the statement of Corollary \ref{cor:4.3} and formula 
(\ref{eq:genus}) the specification of the (weighted homogeneous) 
analytic structure of $(X,0)$ is irrelevant. Indeed, 
for any weighted homogeneous surface singularities with $\mathbb{Q}HS^3$ link, the equivariant geometric genus is topological, it depends only on $\Gamma$ (cf. \cite{NNII, nbook}). 
\end{rem}

%\begin{rem}\label{rem:wc} We note also that the quantity $\chi_{[Z_K]+h_0}(Z_K+E_0^*-s_{[Z_K]+h_0})$ plays a role from the point of view of the reduction theory of lattice cohomology \cite{LNred} as well. A short explanation of this fact would be as follows. 
%For simplicity, let us denote our integral cycle by $l_c:=Z_K+E_0^*-s_{[Z_K]+h_0}$.
%Then, as we know from Remark \ref{rem:delta}(2), the $E_0$-coefficient $c:=m_{E_0}(l_c)$ is the concuctor of the corresponding representable semigroup. Based on the properties of the cannonical equivariant resolution, we can check easily that $(Z_K+E_0^*,E_v)\leq 0$, ie. it is part of the Lipman cone $\cS'_{[Z_K]+h_0}$. This implies that $l_c$ satisfies the properties $m_{E_0}(l_c)=c$ and $(l_c+s_{[Z_K]+h_0})\leq 0$ for any $v\neq v_0$ and in fact, one can prove that it is minimal with respect to these properties. In particular, $l_c\geq 0$ and, moreover, $l_c$ is one of the special cycles which determine the one-dimensional reduced lattice, in which the weight of $c\in \mathbb{Z}_{\geq 0}$ is given by $\chi_{[Z_K]+h_0}(l_c)$ (see e.g. \cite{LNred,stronglyflat} for more details). 
%\end{rem}

\begin{exam}
In this example, we will test (\ref{eq:genus}) on the following canonical equivariant resolution graph:
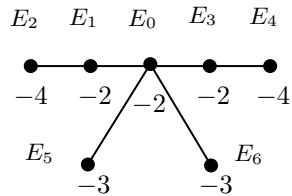
\begin{center}
\tikzset{every picture/.style={line width=0.75pt}}      
\begin{tikzpicture}[x=0.75pt,y=0.75pt,yscale=-1,xscale=1]

%Shape: Circle [id:dp005618189183751676] 
\draw  [fill={rgb, 255:red, 0; green, 0; blue, 0 }  ,fill opacity=1 ] (268,120) .. controls (268,118.34) and (269.34,117) .. (271,117) .. controls (272.66,117) and (274,118.34) .. (274,120) .. controls (274,121.66) and (272.66,123) .. (271,123) .. controls (269.34,123) and (268,121.66) .. (268,120) -- cycle ;
%Shape: Circle [id:dp05309235825097036] 
\draw  [fill={rgb, 255:red, 0; green, 0; blue, 0 }  ,fill opacity=1 ] (297.83,120) .. controls (297.83,118.34) and (299.18,117) .. (300.83,117) .. controls (302.49,117) and (303.83,118.34) .. (303.83,120) .. controls (303.83,121.66) and (302.49,123) .. (300.83,123) .. controls (299.18,123) and (297.83,121.66) .. (297.83,120) -- cycle ;
%Straight Lines [id:da12606964992473801] 
\draw    (274,120) -- (390.5,120) ;
%Shape: Circle [id:dp3988312636022011] 
\draw  [fill={rgb, 255:red, 0; green, 0; blue, 0 }  ,fill opacity=1 ] (357.33,120) .. controls (357.33,118.34) and (358.68,117) .. (360.33,117) .. controls (361.99,117) and (363.33,118.34) .. (363.33,120) .. controls (363.33,121.66) and (361.99,123) .. (360.33,123) .. controls (358.68,123) and (357.33,121.66) .. (357.33,120) -- cycle ;
%Shape: Circle [id:dp8774080700961246] 
\draw  [fill={rgb, 255:red, 0; green, 0; blue, 0 }  ,fill opacity=1 ] (327.67,119) .. controls (327.67,117.34) and (329.01,116) .. (330.67,116) .. controls (332.32,116) and (333.67,117.34) .. (333.67,119) .. controls (333.67,120.66) and (332.32,122) .. (330.67,122) .. controls (329.01,122) and (327.67,120.66) .. (327.67,119) -- cycle ;
%Straight Lines [id:da40089996043415344] 
\draw    (330.67,119) -- (361,170) ;
%Straight Lines [id:da6495415962061992] 
\draw    (330.67,119) -- (299.5,169.5) ;
%Shape: Circle [id:dp18975066670334284] 
\draw  [fill={rgb, 255:red, 0; green, 0; blue, 0 }  ,fill opacity=1 ] (296.5,169.5) .. controls (296.5,167.84) and (297.84,166.5) .. (299.5,166.5) .. controls (301.16,166.5) and (302.5,167.84) .. (302.5,169.5) .. controls (302.5,171.16) and (301.16,172.5) .. (299.5,172.5) .. controls (297.84,172.5) and (296.5,171.16) .. (296.5,169.5) -- cycle ;
%Shape: Circle [id:dp5753141700873128] 
\draw  [fill={rgb, 255:red, 0; green, 0; blue, 0 }  ,fill opacity=1 ] (387.5,120) .. controls (387.5,118.34) and (388.84,117) .. (390.5,117) .. controls (392.16,117) and (393.5,118.34) .. (393.5,120) .. controls (393.5,121.66) and (392.16,123) .. (390.5,123) .. controls (388.84,123) and (387.5,121.66) .. (387.5,120) -- cycle ;
%Shape: Circle [id:dp06268640059520592] 
\draw  [fill={rgb, 255:red, 0; green, 0; blue, 0 }  ,fill opacity=1 ] (358,170) .. controls (358,168.34) and (359.34,167) .. (361,167) .. controls (362.66,167) and (364,168.34) .. (364,170) .. controls (364,171.66) and (362.66,173) .. (361,173) .. controls (359.34,173) and (358,171.66) .. (358,170) -- cycle ;

% Text Node
\draw (261,129.17) node [anchor=north west][inner sep=0.75pt]  [font=\footnotesize] [align=left] {$-4$};
% Text Node
\draw (293,129) node [anchor=north west][inner sep=0.75pt]  [font=\footnotesize] [align=left] {$-2$};
% Text Node
\draw (320,133) node [anchor=north west][inner sep=0.75pt]  [font=\footnotesize] [align=left] {$-2$};
% Text Node
\draw (352,129) node [anchor=north west][inner sep=0.75pt]  [font=\footnotesize] [align=left] {$-2$};
% Text Node
\draw (292.33,174) node [anchor=north west][inner sep=0.75pt]   [align=left] {{\footnotesize $-3$}};
% Text Node
\draw (353.33,174) node [anchor=north west][inner sep=0.75pt]   [align=left] {{\footnotesize $-3$}};
% Text Node
\draw (382,128.83) node [anchor=north west][inner sep=0.75pt]  [font=\footnotesize] [align=left] {$-4$};
% Text Node
\draw (259,90) node [anchor=north west][inner sep=0.75pt]   [align=left] {{\scriptsize $E_2$}};
% Text Node
\draw (288.5,89.5) node [anchor=north west][inner sep=0.75pt]   [align=left] {{\scriptsize $E_1$}};
% Text Node
\draw (318,90) node [anchor=north west][inner sep=0.75pt]   [align=left] {{\scriptsize $E_0$}};
% Text Node
\draw (348,88.5) node [anchor=north west][inner sep=0.75pt]   [align=left] {{\scriptsize $E_3$}};
% Text Node
\draw (378.5,89) node [anchor=north west][inner sep=0.75pt]   [align=left] {{\scriptsize $E_4$}};
% Text Node
\draw (266.5,157.5) node [anchor=north west][inner sep=0.75pt]   [align=left] {{\scriptsize $E_5$}};
% Text Node
\draw (371,158) node [anchor=north west][inner sep=0.75pt]   [align=left] {{\scriptsize $E_6$}};
\end{tikzpicture}
\end{center}
Let $\cS$ be the associated semigroup. Note that the above graph is a `doubled graph' in the sense of \cite{flat}, and $\cS$ is the same as the one associated with the Seifert structure $Sf=(-1; (3,1),(7,4))$. In fact, this is exactly the  semigroup $G(3,5,7)$ generated by $3$, $5$ and $7$. The set of gaps is $\{1,2,4\}$, hence the genus is equal to $3$, a fact that
will be verified next using the formula (\ref{eq:genus}).

Let us denote a cycle by $l'=(l'_0,\dots, l'_6)$ where $l'_i$ is the $E_i$-coefficient. Then the important cycles are as follows: \\
$E_0^*=(21/4,3,3/4,3,3/4,7/4,7/4)$, $Z_K=(13/2,4,3/2,4,3/2,5/2,5/2)$, \\
$r_{[Z_K]+h_0}=(3/4,0,1/4,0,1/4,1/4,1/4)$ and as the result of the generalized Laufer's algorithm (cf. \ref{ss:HrepLip}) we get $s_{[Z_K]+h_0}=(27/4,4,5/4,4,5/4,9/4,9/4)$. Moreover, this also implies $\chi(s_{[Z_K]+h_0})-\chi(r_{[Z_K]+h_0})=0$. In fact, $s_{[Z_K]+h_0}=E^*_0+E^*_2+E^*_4$, thus it will be also useful to have $E_2^*=(3/4,4/7,11/28,3/7,3/28,1/4,1/4)$ and similarly $E_4^*=(3/4,3/7,3/28,4/7,11/28,1/4,1/4)$. 

The explicit calculation of the equivariant geometric genera using, e.g., the formulas of \cite[sect. 6]{LN14} (see also \cite[pg. 241]{nbook}) provides $(p_g)_{[Z_K+h_0]}=0$ and $p_g=3$. Hence, $(X,0)$ is not rational. 

Therefore, according to (\ref{eq:genus}) the genus of the semigroup is given by $\chi_{[Z_K]+h_0}(Z_K+E_0^*-s_{[Z_K]+h_0})-3$, where 
\begin{multline*}
\chi_{[Z_K]+h_0}(Z_K+E_0^*-s_{[Z_K]+h_0})=(-2E_0^*-E_2^*-E_4^*,Z_K-E_2^*-E_4^*)/2\\
=(2(Z_K-E_2^*-E_4^*)_0+(Z_K-E_2^*-E_4^*)_2+(Z_K-E_2^*-E_4^*)_4)/2=(10+1+1)/2=6. 
\end{multline*}
Hence, we get $g(\cS)=3$.
\end{exam}

\section{When is a representable semigroups symmetric?}

\subsection{} 
Let us consider the situation and  notations of section \ref{s:genorb}. Additionally, 
However, we will  introduce some further invariants of Seifert 3-manifolds.

Let $Sf=(-b_0;(\alpha_i,\omega_i)_{i=1}^d)$ be the Seifert invariants of $(X,0)$ and $\Gamma$ be the corresponding canonical equivariant resolution graph. The {\it orbifold Euler number}
is defined as $e=-b_0+\sum_i\omega_i/\alpha_i$. The negative definiteness of the intersection form is equivalent with  $e<0$. We also write  $\alpha:=\mathrm{lcm}(\alpha_1,\ldots,\alpha_d)$. Let $|H|$ be the order of $H$, and
let  $\mathfrak{o}$ be the order of the class of $E^*_0$ in $H$. 
Then, by \cite{neumann.abel}, 
\begin{equation}\label{eq:sei2}
 |H|=\alpha_1\cdots\alpha_d|e|, \ \ \ \ \mathfrak{o}=\alpha|e|.
\end{equation}
Let $\gamma$ be the $E_0$-coefficient of the rational cycle $Z_K-E$. This combinatorial number has several interpretations -- eg. exponent of the weighted homogeneous germ, log-discrepancy of $E_0$, Goto-Watanabe $a$-invariant, orbifold Euler characteristic, etc. -- since it has a central importance with respect to the properties of $(X,0)$. It also equals 
$$\gamma=\frac{1}{|e|}\cdot \Big( d-2-\sum_{i=1}^d \frac{1}{\alpha_i}\Big)\in \mathbb{Q}.$$

\subsection{} In subsection \ref{ss:4.2} we already mentioned that the semigroup $\cS_{\Gamma}$ associated with $(X,0)$ and with  the central vertex $E_0$ (or, with the $\mathbb{C}^*$-action) 
is topological: it can be determined completely from $\Gamma$. 

Let $c$ be the conductor of $\cS_{\Gamma}$, that is, $c$ is the smallest integer such that 
$c+\mathbb{N}\subset \cS_\Gamma$. A numerical semigroup $\cS\subset \mathbb{N}$ is {\it symmetric}
whenever  $s\in \cS$ if and only if $c-1-s\not\in\cS$. 

In the sequel, we wish to detect topologically whether $\cS_\Gamma$ is symmetric or not.

Recall that by Theorem \ref{mainsmgps} we have $\cS_\Gamma=\cS_{C^{gen}}$, that is, $\cS_\Gamma$
is realized as the semigroup of an abstract reduced irreducible curve singularity. 
In such a generality, we have the following: For an irreducible curve singularity $(C,o)$ with semigroup $\cS$  and delta invariant $\delta$ the following facts are equivalent \cite{Kunz, Delgado}:
\begin{equation}\label{eq:GORE} \mbox{$\cS$ is symmetric \ $\Leftrightarrow$ \ $(C,o)$ is Gorenstein  \ $\Leftrightarrow$ \
$c=2\delta$}. \end{equation}
Having in mind the identity  $\cS_\Gamma=\cS_{C^{gen}}$, we can ask if the symmetry of $\cS_\Gamma$ can be related with the Gorenstein property of $(X,0)$. The answer in general is no.

Indeed, the Gorenstein property of $(X,0)$ does not imply the symmetry of $\cS_\Gamma$: 
 Example 7.1.3 of \cite{stronglyflat} 
 is a Gorenstein singularity with Seifert invariants $$(-2; (2,1), (2,1), (3,1),(3,1), (7,1), (7,1), (84,1))$$ whose semigroup is not symmetric. 

On the other hand, a singularity with Seifert invariants $(-2; (3,1), (3,1), (3,1))$
 has a symmetric semigroup, namely $G(2,3)$, however, it is not even  numerically Gorenstein. 
% In particular, $(X,0) $ being a numerically Gorenstein singularity
% (a topological property of $\Gamma$) is not enough to guarantee the symmetry of $\cS_\Gamma$.

\subsection{} In the next characterization, we will use the last property of (\ref{eq:GORE}), 
applied for $(C^{gen},0)$.

 By \cite{stronglyflat} we know that the conductor $c$ of $\cS_\Gamma$ can be expressed as 
\begin{equation}\label{eq:COND}
    c=(Z_K+E_0^*-s_{[Z_K]+h_0})_0=\gamma+1+\frac{1}{|e|}-\check{s},
\end{equation}
where $( \ )_{0}$ refers to the $E_0$-coefficient of the corresponding cycle, and $\check{s}:=(s_{[Z_K]+h_0})_0$.

Now, using (\ref{eq:COND}) and $\chi(E^*_0)=-(E^*_0, E^*_0-Z_K)/2= 
-Z_{K,0}/2+1/(2|e|)$ we get 
\begin{equation}\begin{split}
\chi_{[Z_K]+h_0}(Z_K+E_0^*-s_{[Z_K]+h_0})&= \chi(Z_K+E^*_0)-\chi(s_{[Z_K+h_0]})\\
&=\frac{c}{2}+\frac{\check{s}}{2}-\chi(s_{[Z_K]+h_0}).
\end{split}
\end{equation}
This rephrases (\ref{eq:genus}) as 
\begin{equation}\label{eq:genus2}
g(\cS_\Gamma)=\frac{c}{2}+\frac{\check{s}}{2}-\chi(r_{[Z_K]+h_0})+(p_g)_{[Z_K]+h_0}-p_g.
\end{equation}
Now, by (\ref{eq:GORE}) $\cS_\Gamma$ is symmetric if and only if $g(\cS_\Gamma)=\delta=c/2$, 
hence (\ref{eq:genus2}) reads as follows. 
\begin{cor} A representable semigroup $\cS_\Gamma$ is symmetric if and only if $\Gamma$ satisfies 
\begin{equation}
p_g-(p_g)_{[Z_K]+h_0}=\frac{\check{s}}{2}-\chi(r_{[Z_K]+h_0}), 
\end{equation}
or, equivalently
\begin{equation}
p_g=h^1(\cO_{{\tilde X}}(-s_{[Z_K]+h}))+\check{s}/2 - \chi(s_{[Z_K]+h_0}).
\end{equation}
\end{cor}

\begin{exam}\label{ex:rat}
Assume that the ambient space $(X,0)$ is a rational singularity with a $\mathbb{C}^*$-action. Let $\Gamma$ be its canonical equivariant resolution graph. In this case, one has $(p_g)_{[Z_K]+h_0}=\chi(r_{[Z_K]+h_0})-\chi(s_{[Z_K]+h_0})$, hence 
\begin{equation}\label{eq:genusrat}
g(\cS_\Gamma)=\chi_{[Z_K]+h_0}(Z_K+E_0^*-s_{[Z_K]+h_0})=\frac{c}{2}+\frac{\check{s}}{2}-\chi(s_{[Z_K]+h_0}).
\end{equation}
In particular, $\cS_\Gamma$ is symmetric if and only if 
 $\chi(s_{[Z_K]+h_0})=\check{s}/2$.
(If $b_0\geq d$ then we know by \cite[6.2.2]{stronglyflat} that $\cS_\Gamma=\mathbb{N}$,
hence $g(\cS_\Gamma)=0$ too.  In case $b_0<d$, the conductor satisfies $c\geq 2$ and the genus is also nonzero.)

In fact, the symmetry (or  $\chi(s_{[Z_K]+h_0})=\check{s}/2$)
is equivalent to $(Z_K+E_0^*-s_{[Z_K]+h_0}, s_{[Z_K]+h_0})=0$. Since $s_{[Z_K]+h_0} \in \cS'_{[Z_K]+h_0}$, and $Z_K+E_0^*-s_{[Z_K]+h_0}$ is positive 
%(Remark \ref{rem:wc}) 
or the zero cycle, we deduce that in this case
\begin{center}
$\cS$ is symmetric if and only if $b_0\geq d$ or $[Z_K+E^*_0]=0$.
\end{center}
In particular, this is the case when we consider any chosen $\mathbb{C}^*$-action on $(\mathbb{C}^2,0)$. In fact, the non-trivial cases are the $G(p,q)$ semigroups (cf. \cite[3.2]{flat}) that are indeed symmetric. 
\end{exam}

In general, from (\ref{eq:genus}) follows that condition $[Z_K+E_0^*]=0$ implies symmetry. In particular, this is the case when $\Gamma$ is numerically Gorenstein (that is, $[Z_K]=0$) and satisfies $\mathfrak{o}=1$, as was shown in \cite[Prop.2.3.9]{flat}. However, the inverse implication is not true. In fact, if we consider the graph associated with $Sf=(-2;2\times (2,1), 2\times (3,1))$, then the semigroup is $G(2,3)$ which is indeed symmetric, but $[Z_K+E_0^*]\neq 0$ in $H$.

%\bibliographystyle{amsplain}
%\bibliography{papers}
\providecommand{\bysame}{\leavevmode\hbox to3em{\hrulefill}\thinspace}
\providecommand{\MR}{\relax\ifhmode\unskip\space\fi MR }
\providecommand{\MRhref}[2]{%
  \href{http://www.ams.org/mathscinet-getitem?mr=#1}{#2}
}

\end{document}